\documentclass[a4paper,12pt]{amsart}
\usepackage{euscript,eufrak,verbatim}
\usepackage[usenames]{color}

\usepackage[final]{epsfig}
\usepackage{psfrag}
    
\def\e{\varepsilon}

\newcommand{\E}{\mathbb E}
\DeclareMathOperator{\diam}{diam}

\DeclareMathOperator{\supp}{supp}

\def\R{\mathbb{R}}

\def\N{\mathbb{N}}

\def\S{\EuFrak{S}}
\def\mL{\EuScript{L}}
\def\hL{\widehat{\EuScript{L}_\phi}}
\def\P{\mathbb{P}}

\def\1{\mathbf 1}

\def\uo{\underline{\omega}}
\def\ux{\underline{x}}

\newtheorem{theorem}{Theorem}[section]
\newtheorem{proposition}[theorem]{Proposition}
\newtheorem{corollary}[theorem]{Corollary}
\newtheorem{remark}[theorem]{Remark}
\newtheorem{definition}[theorem]{Definition}
\newtheorem{example}[theorem]{Example}

\DeclareMathSymbol{\varnothing}{\mathord}{AMSb}{"3F}

\begin{document}

\renewcommand{\theequation}{\thesection$\cdot$\arabic{equation}}

\title{On Transfer Operators and Maps with Random Holes}
\author{Wael Bahsoun} \author{J\"org Schmeling} \author{Sandro Vaienti}

\address{W.\  Bahsoun: Department of Mathematical Sciences, Loughborough University,
Loughborough, Leicestershire, LE11 3TU, UK}
\email{W.Bahsoun@lboro.ac.uk}
\address{J.\ Schmeling: Mathematics Centre for Mathematical Sciences,
Lund Institute of Technology, Lund University Box 118 SE-221 00 Lund, Sweden}
\email{joerg@maths.lth.se}

\address{S.\ Vaienti: Aix Marseille Universit\'e, CNRS, CPT, UMR 7332, 13288 Marseille, France and
Universit\'e de Toulon, CNRS, CPT, UMR 7332, 83957 La Garde, France. }
\email{vaienti@univ-mrs.fr}

\date{\today}
\thanks{WB and JS would like to thank the University of Toulon and CPT Marseille where this work was initiated and finished. WB, JS and SV would like to thank the Isaac Newton Institute for Mathematical Sciences where parts of this research was conducted during the programme Mathematics for the Fluid Earth, 21 October to 20 December 2013. SV thanks the French ANR Perturbations,  the CNRS-PEPS ``Mathematical Methods for Climate Models'' and the CNRS Pics ``Propri\'et\'es statistiques des syst\'emes d\'eterministes et al\'eatoires" N. 05968, for support. SV thanks GDR ``Analyse Multifractal'' for supporting his visit to Lund University.}
\begin{abstract}
We study Markov interval maps with{ \em random holes}. The holes are not necessarily elements of the Markov partition. Under a suitable, and physically relevant, assumption on the noise, we show that the transfer operator associated with the random open system can be reduced to a transfer operator associated with the closed deterministic system. Exploiting this fact, we show that the random open system admits a unique (meaningful) absolutely continuous conditionally stationary measure. Moreover, we prove the existence of a unique probability equilibrium measure supported on the survival set, and we study its Hausdorff dimension.
\end{abstract}
\maketitle


\section{Introduction}

A dynamical system is called open if there is a subset in the phase space, called a \textit{hole}, such that whenever an orbit lands in it, the dynamics of this orbit is terminated. Statistical aspects of such open systems have been addressed by many authors, see for instance \cite{Alt,DY,Det}. Open dynamical systems have found interesting applications in physics \cite{Alt, GG, WBW}, and more recently, after the pioneering work of \cite{BY, KL2}, it was found that open systems are intimately related to studying  metastable dynamical systems \cite{BHV, BV1,BV2, Det1, DW, FS, GHW} and their applications in geophysical sciences \cite{DFP, FP}.

\bigskip

Following the work of \cite{BV2}, in this paper we study dynamical systems with {\em random holes}, called random open systems. One of the main motivation for studying random open systems is that they contribute to understanding the long-term statistics of random metastable systems. In \cite{BV2} random perturbations of interval maps that initially admit exactly two invariant ergodic densities were studied. Under random perturbations, which generate \textit{random holes}, leakage of mass between the two initially ergodic subsystems forces the random system to mix and admit a unique invariant density. It was shown in \cite{BV2} that the invariant density of the random system can be approximated in the $L^1$ norm (with respect to Lebesgue measure),  by a particular convex combination of the two invariant ergodic densities of the initial system. In particular, the weights in the convex combination is identified as the ratio of the escape rates from the left and right random open systems.

\bigskip

However, although almost every point (with respect to Lebesgue) from each initially ergodic subsystem escapes to the other ergodic component, some points {\em survive} in their initial set and do not visit the other ergodic component.

\bigskip

The main motivation of this paper is to study statistical properties of orbits which survive escaping from a random open system, and to determine the Hausdorff dimension of their set. From applications point of view, one can use our results to study {\em statistical properties and dimension theory} of orbits that do not visit other ergodic components in a random metastable system, such as the ones studied in \cite{BV2}, and to provide insight to geophysical models that study regions of the ocean with slow or poor mixing properties \cite{DFP}. Mathematically, previous results on the Hausdorff dimension of the survival set were obtained in \cite{GS, FP1, DT, LM, S} for maps with \textit{deterministic} holes.  In our \textit{random} setting, an interesting feature of our current work is that it uses tools from deterministic and closed systems to obtain results in random open systems: the transfer operator associated with our random open system can be reduced to a transfer operator associated with the closed deterministic system.

\bigskip

In section 2 we introduce the class of maps that we study, the associated transfer operator and the space of functions that it acts on, and known results from deterministic closed systems. In section 3 we introduce a class of maps with random holes, the notion of a conditionally stationary measure and prove a theorem highlighting its significance. Section 4 contains a topological characterization of the survival set of the random open system. In section 5 we introduce the transfer operator associated with the random system and show how to reduce it to a transfer operator of the original closed deterministic system. Exploiting this fact, we show that the random open system admits a unique (meaningful) absolutely continuous conditionally invariant measure. Moreover, we prove the existence of a unique probability equilibrium measure supported on the survival set, and we study its Hausdorff dimension. Section 6 contains two simple examples that highlight our results.

\section{Setup and Preliminaries}
\subsection{Class of maps}\label{maps}
Let $I:=[0,1]$ and $T:I\to I$ be a Markov map; i.e., there exists a finite partition of open intervals $\mathcal P=\{\mathcal P_i\}_{i=1}^{l}$ such that:
\begin{enumerate}
\item$ I=\cup_{i=1}^{l}\overline{\mathcal P_i}$, and $\mathcal P_i\cap \mathcal P_j=\emptyset$ when $i\not=j$;
\item $\exists$ $0<\alpha<1$ such that $T_{|P_i}$ is $C^{1+\alpha}$;
\item if $T(\mathcal P_i)\cap\mathcal P_j\not=\emptyset$ then $\mathcal P_j\subset\overline{T(\mathcal P_i)}$.
\item $\sup_{P\in\mathcal P^{(n)}}\text{diam}(P)\to 0$ as $n\to\infty$, where $\mathcal P^{(n)}:=\vee_{i=0}^{n-1}T^{-i}\mathcal P$.\\

\noindent We further assume that
\item $\exists\, N\in\mathbb N$ such that, for all $i=1,\dots l$, $\overline{T^N(\mathcal P_i)}=I.$
\end{enumerate}
\begin{remark}
Under assumptions (1)-(4), it is well know that there exists a semi-conjugacy between a one-sided subshift of finite type $(\Sigma, \sigma)$ and $([0,1], T)$. To avoid using $(\Sigma, \sigma)$ and to keep our presentation mainly on $[0,1]$, we present some of the arguments using the map $T(x)=2x$ mod $1$. All the results of this paper are true for the class of maps introduced above in subsection \ref{maps}.
\end{remark}
\subsection{Space of functions and the transfer operator} We define our space of functions following the prescriptions at the beginning of Sect. 1 in \cite{AD}. Let ${\mathcal J}:=\{T{\mathcal (P_i)}\}_{i=1}^l.$ Since ${\mathcal J}$ is contained in the sigma-algebra $\sigma({\mathcal P})$ generated by ${\mathcal P},$ there exists a coarser partition ${\mathcal G}$ such that $\sigma({\mathcal J})=\sigma({\mathcal G}).$ Let
$$
{\mathcal H}_{\alpha}^{loc}:=\{\phi: I\rightarrow \mathbb{R}|\ \exists C>0, s.t. \ \forall x,y\in G_j, G_j\in {\mathcal G}, \ |\phi(x)-\phi(y)|\le C|x-y|^{\alpha}\}
$$
For $\phi\in\mathcal H^{\text{loc}}_{\alpha}$ we let $C_{\alpha}(\phi)$ denote the semi-norm
$$C_{\alpha}(\phi)=\sup_{\underset{G_j\in\mathcal G}{x\not=y\in G_j}}\frac{|\phi(x)-\phi(y)|}{|x-y|^{\alpha}}.$$
When equipped with the norm $||\cdot||_{\mathcal H^{\text{loc}}_{\alpha}}:=C_{\alpha}(\cdot)+||\cdot||_{\infty}$, $\mathcal H^{\text{loc}}_{\alpha}$ is the Banach space of locally H\"older continuous functions with exponent $\alpha$. If a function $f:I\to\mathbb R$ is H\"older with exponent $\alpha$, on the whole interval $I$, we simply write $f\in\mathcal H_{\alpha}$.\\

Let $\phi\in\mathcal H^{\text{loc}}_{\alpha}$. The transfer operator, associated with $T$ and with potential $\phi$, acting on $\mathcal H^{\text{loc}}_{\alpha}$  is defined as
\[
\mL_\phi(f)(x):=\sum_{Ty=x}e^{\phi(y)}f(y).
\]
 Let $\sigma(\mathcal L_{\phi})$ denote the spectrum of $\mathcal L_{\phi}$ as an operator on $\mathcal H^{\text{loc}}_{\alpha}$, and let $\mL^*_{\phi}$ denote the dual operator of $\mL_{\phi}$.
 \subsection{Conformal measure}
 Let $A$ be a Borel measurable set. A measure $m$ is called $\phi$-conformal if
 $$m(T(A))=\int_Ae^{-\phi}dm$$
 whenever $T:A\to T(A)$ is injective.
\subsection{Known results about the deterministic system $T$}
The following is well known result (see for instance \cite{Bo, Ba})
\begin{proposition}\label{prop0}
Let $T$ satisfy assumptions (1)-(5). The following holds:
\begin{itemize}
\item[(i)] $\mL_\phi:\mathcal H^{\text{loc}}_{\alpha}\to\mathcal H^{\text{loc}}_{\alpha}$ has a dominant simple eigenvalue $\lambda$ and its corresponding eigenfunction, $\rho$, is strictly positive.
\item[(ii)] $\sigma(\mL_{\phi})\setminus \{\lambda\}\subset B(0, r)$, with $r<\lambda$.
\item [(iii)] There is a unique probability measure $\nu$ such that $\mL^*_{\phi}\nu=\nu$;
\item[(iv)] For all $f\in \mathcal H^{\text{loc}}_{\alpha}$ we have
$$\lim_{n\to\infty}||\lambda^{-n}\mL^n_{\phi}f-\rho\int_I fd\nu||_{\infty}=0,$$
and the probability measure $\mu:=g\nu$ is an equilibrium state associated with $\phi$.
\end{itemize}
\end{proposition}
It is well known that the measure $\nu$ is also $\phi$-conformal.
\section{Perturbations and random holes}
Let $(\omega_k)_{k\in\mathbb N}$ be an i.i.d.  stochastic process with values in the interval $S:=[0,1/2]$, and with probability distribution $\theta$; we will set $\overline{S}:=S^{\otimes \mathbb{N}}$ for the direct product of $S$ upon which the direct product measure $\theta^{\otimes \mathbb{N}}$ is defined.  We fix a point $x_0\in [0,1]$ and consider ``random'' holes around it. More precisely, for $\omega\in [0,1/2]$ we associate the interval $I_\omega:=(x_0-\omega, x_0+\omega)$. We call such an interval a random {\em hole}.\\

To explain the dynamics of the system with random holes, consider a finite path of the stochastic process, say, $(\omega_0,\omega_1,\dots, \omega_k)$. For this path, we first restrict $T$ to
$I^c_{\omega_o}$, then any point in $I_{\omega_o}^c$ that gets mapped by $T$ into $I_{\omega_1}$, its orbit gets terminated. At time $n=2$, any $x\in \left(I^c_{\omega_0}\cap T^{-1}(I^c_{\omega_1})\right)$ that gets mapped by $T$ into $I_{\omega_2}$ its orbit gets terminated, and so on. The following measure, which was first introduced in \cite{PY} in the determinstic setting and in \cite{BV2} in the random setting, plays a central role in our analysis:
\begin{definition}
A Borel measure $\hat{\alpha}$ on $[0,1]$ satisfying
\begin{equation*}
\hat{\alpha}(A)=\frac{1}{\hat{\lambda}}\int_{S} d\theta(\omega)\hat{\alpha}(T^{-1}A\cap T^{-1}I^c_{\omega});
\end{equation*}
where
$$\hat{\lambda}=\int_S d\theta(\omega)\hat{\alpha}(T^{-1}I_{\omega}^c),$$
is called {\bf conditionally stationary measure}. The \text{escape rate}, with respect to $\hat{\alpha}$, from the random open system is given by $-\ln\hat\lambda$.
Moreover, we have
\begin{equation}\label{eq2}
\hat\lambda^n\hat\alpha(A)=\int_{\bar S}  d\theta^{\infty}(\bar{\omega})\hat\alpha(T^{-n}A\cap T^{-1}I^c_{\omega_1}\cap T^{-2}I^c_{\omega_2}\cap\cdots\cap T^{-n}I^c_{\omega_n}).
\end{equation}
\end{definition}
\subsection{Significance of conditionally stationary measures}
Consider the following random variable:
$$\tau_{\bar\omega}(x)=\sup\{i\ge0:\, T^i(x)\in I^c_{\omega_i}\}.$$
If $\tau_{\bar\omega}(x)=n$, it means that, given the path $\bar\omega$, the orbit of $x$ escapes through a random hole, $I_{\omega_{n+1}}$, exactly at time $n+1$.  We are interested in estimating the following expectation
\begin{equation}\label{eq3}
\E_{\theta}[\E_m[\tau_{\bar\omega}(x)]].
\end{equation}
Note that the first expectation is taken with respect to the ambient measure $m$. If the expectation in \eqref{eq3} is finite, it means that for almost every path $\bar \omega$, $m$-almost every point $x$ will not survive.\\

The following theorem provides a useful estimate to \eqref{eq3}. In particular it shows that whenever $\hat\alpha$ is a conditionally invariant measure which is equivalent to $m$, for almost every path $\bar \omega$, $m$-almost every point $x$ will not survive.
\begin{theorem}\label{prop1}
If $\hat\alpha$ is a conditionally stationary measure then:
\begin{enumerate}
\item $\E_{\theta}[\E_{\hat\alpha}[\tau_{\bar\omega}(x)]]=\frac{\hat\lambda}{1-\hat\lambda};$
\item $\E_{\theta}[\E_{\hat\alpha}[\tau_{\bar\omega}(x)]^2]-(\E_{\theta}[\E_{\hat\alpha}[\tau_{\bar\omega}(x)]])^2=\frac{\hat\lambda}{(1-\hat\lambda)^2}.$
\end{enumerate}
\end{theorem}
\begin{proof}
Given a path $\bar\omega$, set $$R_{n,\bar\omega}:=T^{-1}I^c_{\omega_1}\cap T^{-2}I^c_{\omega_2}\cap\cdots\cap T^{-n}I^c_{\omega_n}.$$
By \eqref{eq2}, we have
\begin{equation}\label{eq4}
\E_{\theta}(\hat\alpha(R_{n,\bar\omega}))=\hat\lambda^n.
\end{equation}
Therefore, by \eqref{eq4}, we have
\begin{equation}
\begin{split}
\E_{\theta}[\E_{\hat\alpha}[\tau_{\bar\omega}(x)]]&=\E_{\theta}[\sum_{n=0}^{\infty}n(\hat\alpha(R_{n,\bar\omega})-\hat\alpha(R_{n+1,\bar\omega}))]\\
&=\sum_{n=0}^{\infty}n[\E_{\theta}(\hat\alpha(R_{n,\bar\omega}))-\E_{\theta}(\hat\alpha(R_{n+1,\bar\omega}))]\\
&=\sum_{n=0}^{\infty}n(\hat\lambda^{n}-\hat\lambda^{n+1})=\frac{\hat\lambda}{1-\hat\lambda}.
\end{split}
\end{equation}
In the above computation we have used Tonelli's Theorem to exchange the sum and the expectation, and in the last step we have used the fact that
$$\sum_{n=0}^{\infty}n\hat\lambda^n=\frac{\hat\lambda}{(1-\hat\lambda)^2}.$$ Similarly, for (2), using the fact that
$$\sum_{n=0}^{\infty}n^2\hat\lambda^n=\hat\lambda\frac{1+\hat\lambda}{(1-\hat\lambda)^3},$$ we have
\begin{equation}
\begin{split}
\E_{\theta}[\E_{\hat\alpha}[\tau_{\bar\omega}(x)]^2]-(\E_{\theta}[\E_{\hat\alpha}[\tau_{\bar\omega}(x)]])^2&=\sum_{n=0}^{\infty}n^2((\hat\lambda^{n}-\hat\lambda^{n+1})-\frac{\hat\lambda^2}{(1-\hat\lambda)^2}\\
&=(1-\hat\lambda)\sum_{n=0}^{\infty}n^2\hat\lambda^n-\frac{\hat\lambda^2}{(1-\hat\lambda)^2}=\frac{\hat\lambda}{(1-\hat\lambda)^2}.
\end{split}
\end{equation}
\end{proof}
\begin{remark}
In \cite{BV2}, it was proved that piecewise expanding maps with random holes admit a conditionally invariant measure which is equivalent to Lebesgue. Moreover, the index of dispersion provided by Theorem \ref{prop1}
$$\frac{\E_{\theta}[\E_{\hat\alpha}[\tau_{\bar\omega}(x)]^2]-(\E_{\theta}[\E_{\hat\alpha}[\tau_{\bar\omega}(x)]])^2}{\E_{\theta}[\E_{\hat\alpha}[\tau_{\bar\omega}(x)]]}$$ plays a crucial role in approximating invariant densities of metastable random maps.
\end{remark}

Conditionally invariant measures are quite tricky. In the case of deterministic holes, Demers and Young provided an example of a map with a hole that has infinite number of absolutely continuous conditionally invariant measures (accim) with overlapping supports \cite{DY}. Moreover, all the accims in the example of \cite{DY} have the same escape rate. However, it is suggested in \cite{DY} that only `natural' accims are meaningful. For instance, `natural' maybe in the sense that the density of the accim belongs to a certain class of functions that include the constant density, such that under the iterates of the conditional transfer operator each function in this class converges in the appropriate topology to the density of the `natural' accim (see section 5.2 of \cite{DY}).\\

In \cite{BV2} we studied maps with random holes in the case where the potential of the transfer operator is $-\ln|T'(x)|\in BV$, where $BV$ is the space of functions of bounded variation on the unit interval. For sufficiently small holes, using the perturbation result of \cite{KL1}, we proved in \cite{BV2} the existence of a unique absolutely continuous conditionally stationary measure whose density belongs to $BV$, and which is natural along the lines above. In the current work, the transfer operator has a general H\"older potential, and there is no smallness condition on the size of the holes. In Theorem \ref{main} and Corollary \ref{cor1} we obtain a unique absolutely continuous conditionally stationary measure that satisfy the criteria of section 5.2 of \cite{DY}.

\vskip 1cm

\section{Topological description of the surviving set}
We now introduce the random function
\[
\1(x,\omega)=\begin{cases} 1 & x\not\in I_\omega \\ 0 & \text{ else} \end{cases}.
\]
Clearly
\begin{equation}\label{trans}
\begin{split}
\P_\theta(\1 (x,\omega)=1)&=\int_{-\infty}^\infty\1 (x,\omega)\, d\theta(\omega)\\&=\theta(|x-x_0|\ge \omega)=\int_0^{|x-x_0|}\, d\theta\\&=F_\theta(|x-x_0|),
\end{split}
\end{equation}
where $F_\theta$ is the distribution function of $\omega$ (We note that here $F_\theta(0)=\theta (\{0\}$).
\vskip 1cm

Let $\uo =\omega_0\omega_1\cdots\omega_n\cdots\in [0,1]^\N$.
We introduce the i.i.d. (for fixed $x$) stochastic functional process
\[
g_n(x,\uo):=\1 (x,\omega_n).
\]
The function $g_n(T^nx,\uo)$ describes whether the trajectory of a point $x$ falls into a hole at time $n$ or not.

We are interested in the set of points that avoid, under the dynamics of $T$, the random holes, i.e. those trajectories that ``survive''. We first start with a topological description.

Given a realization $\uo\in [0,1]^\N$ we define its surviving set as
\[
\S_{\uo}:=\left\{x\in [0,1]\, :\, g_n(T^nx,\omega)=1 \,\, \forall n\in\N\right\}.
\]
The potential surviving set is defined as
\[
\S:=[0,1]\setminus\{x\in [0,1]\, :\, x\not\in\S_{\uo}\, \text{ for $\theta^{\otimes\N}$ - a.e. } \uo\}.
\]

\begin{theorem}
Let $\supp \theta =[a,b]$. Then
\[
\S=\bigcup_{m\in\N}T^{-m}\left(\bigcap_{n\in\N}T^{-n}([0,1]\setminus (x_0-b,x_0+b))\right).
\]
\end{theorem}
\begin{proof}
If a point's orbit visits the interval $(x_0-b,x_0+b)$ only finitely many times its surviving probability is positive. Hence the inclusion ''$\supseteq$'' is clear.

Fix $\e>0$ and $x\in [0,1]$. If the trajectory of $x$ hits the interval $(x_0-b+\e,x_0+b-\e)$ infinitely often the Borel Cantelli Lemma implies that it will eventually fall into a hole since $\P_\theta(g_n(T^nx,\uo)=0)$ is, uniformly in $n$, bounded away from zero as long as $|T^nx-x_0|<b-\e$. This proves ''$\subseteq$''.
\end{proof}
\begin{remark}\label{non-empty}
For $T(x)=2x$ $\text{mod }1$, it is easy to check that if $b<1/8$, then the potential surviving set is non-empty.
In fact if $b<1/8$ the hole $(x_0-b,x_0+b)$ is contained in two adjacent binary intervals of length $1/4$. In the binary coding these two intervals correspond (for each interval separately) to fixing the first two digits, i.e. a cylinder of length 2. Now the set of symbolic sequences $\left\{\ux\in\{0,1\}^\N\right\}$ that do not contain any two a priori fixed words out of $[00], [01], [10], [11]$ as a sub-word contains the sequences $0^\infty$, $1^\infty$ or $(01)^\infty$.
\end{remark}


\section{Statistical aspects of the surviving trajectories}

To study the long-term statistics of the surviving trajectories, we first introduce the averaged transfer operator associated with the system with random holes. We will introduce in this section a interesting condition\footnote{The justification and the interpretation of this condition will be discussed below in this section. See subsection \ref{inter}.} that allows us to reduce the transfer operator associated with the system with random holes to a transfer operator (with a new potential) associated with the deterministic and closed system $T$.

\subsection{The averaged transfer operator} In random dynamical systems \cite{K}, it is often useful to study the average transfer operator of the random system, which is the integral, with respect to the noise, of the transfer operators associated with the perturbed maps. Here in our setting, the map is fixed at all times. However, a hole is randomly selected at each time step.
\vskip 1cm
Let $\phi\in\mathcal H^{\text{loc}}_{\alpha}$. The averaged transfer operator with potential $\phi$, associated with the system with random holes and acting on $\mathcal H^{\text{loc}}_{\alpha}$, is defined by:
\[
\hL (f)(x):=\int_S\mL_{ \phi}(f\cdot \1(\cdot, \omega))(x)\, d\theta(\omega).
\]
Its dual operator, which acts on the space of finite measures, will be denoted by $\hL^\ast$. Notice that without further assumptions on the hole $I_{\omega}$ the operator $\hL (f)$ will not leave the space $\mathcal H^{\text{loc}}_{\alpha}$ invariant. We provide in this way; we first
recall the definition of $\P_{\theta}$ in \eqref{trans} and we  write
$$g(x):=\P_{\theta}\left(\1(x, \omega)=1\right)\text{ and } \psi(x):=\log g(x).$$ Then we will make the following \\
{\bf Assumption:} $\psi\in\mathcal H_{\alpha}$.\\

\subsection{When does $\psi\in\mathcal H_{\alpha}$? And what does it physically mean?}\label{inter}
\begin{definition}
A measure $\theta$ on $[0,+\infty)$ is called {\em Ahlfohrs upper semi-regular} if there is a constant $K>0$ and a real number $\alpha>0$ such that for all non-empty {\bf open} intervals $I\subset{[0,+\infty)}$
\[
\frac{\theta (I)}{(\diam I)^\alpha}<K.
\]
\end{definition}

\begin{remark}
Ahlfors upper semi-regular measures include fractal measures like the measures of maximal dimension of dynamically defined Cantor sets, i.e. Cantor sets that arise from smooth expanding repellers.
\end{remark}

\begin{theorem}
The function $g$ is H\"older continuous if $\theta$ is Ahlfors upper semi-regular. If in addition $\theta(\{0\})>0$ then $\psi$ is H\"older continuous.
\end{theorem}
\begin{proof}
Let $x,y\ne x_0$ and $I_{x,y}=(|x-x_0|,|y-x_0|)\subset\R^+$. Then $\diam I_{x,y}\le |x-y|$. If $\theta$ is Ahlfors upper semi-regular we conclude
\begin{align*}
|g(x)-g(y)|&=|F_\theta(|x-x_0|)-F_\theta(|y-x_0|)=\left|\int_0^{|x-x_0|}\, d\theta -\int_0^{|y-x_0|}\, d\theta\right|\\&
=\left|\int_{|y-x_0|}^{|x-x_0|}\, d\theta\right|=\theta (I_{x,y})\le K(\diam I_{x,y})^\alpha\\&
\le K|x-y|^\alpha.
\end{align*}
If $x=x_0$ then
\begin{align*}
|g(x_0)-g(y)|&=\left|\theta (\{0\})-\int_0^{|y-x_0|}\, d\theta\right|\\&
=\left|\theta (\{0\})- \left(\theta (0)+\lim_{\e\to 0^+}\int_\e^{|y-x_0|}\, d\theta\right)\right|\\&
=\lim_{\e\to 0^+}\theta ((\e,|y-x|)\le K|x-y|^\alpha.
\end{align*}
If in addition $\theta (\{0\})>0$ we have that $F_\theta(|x-x_0|)>0$ for all $x\in [0,1]$ and hence $g>0$. That together with the H\"older continuity of $g$ implies that $\psi=\log g$ is H\"older continuous.
\end{proof}
\begin{remark}
$\theta (\{0\})>0$ has a physical interpretation. It means that the system randomly switches on a random hole and it has the possibility of not switching on any hole at all. More precisely, at any moment $n$ the system is closed (no trajectory falls into a hole) with positive probability.
\end{remark}
\subsection{Conditionally invariant measures and equilibrium states}
\begin{theorem}\label{main} Assume $\theta$ is Ahlfohrs upper semi-regular. For $\psi \in\mathcal H^{\text{loc}}_{\alpha}$ we have
\begin{itemize}
\item[(i)] For $f\in\mathcal H^{\text{loc}}_{\alpha}$ we have
$\hL (f)=\mL_\phi(f\cdot g)$ and \[
\hL (f)=\mL_{\phi +\psi} (f).
\]
\item[(ii)] $\supp g=[0,1]\setminus (x_0-a,x_0+a)$ where $\supp\theta=[a,b]$.
\item[(iii)] $\hL:\mathcal H^{\text{loc}}_{\alpha}\to\mathcal H^{\text{loc}}_{\alpha}$ has a dominant simple eigenvalue $\hat\lambda$ and its corresponding eigenfunction, $\hat\rho$, is strictly positive. There is a unique eigenmeasure $\hat\nu$ with $\hL^\ast(\hat\nu)=\hat\lambda \hat\nu$. The {\em associated measure} is given by $\hat\mu_{\phi,\theta}:=\hat\rho\hat\nu$.
\item[(iv)] $\supp\hat\mu_{\phi,\theta}=\bigcap_{n\in\N}T^{-n}\left([0,1]\setminus (x_0-a,x_0+a)\right)$
\item[(v)] The associated equilibrium state to the potential $\phi +\psi$ is fully supported.
\end{itemize}
\end{theorem}
\begin{proof}
\begin{itemize}
\item[(i)] By definition we have
\begin{align*}
\hL (f)&=
\int_S\sum_{Ty=x}e^{\phi(y)}f(y)\1 (y,\omega)\, d\theta(\omega)
=\sum_{Ty=x}e^{\phi(y)}f(y)\mathbb{P}_{\theta}({\bf 1}(y,\omega)=1)\\&
=\mL_\phi(f\cdot g)=\mL_{\phi+\psi}(f).
\end{align*}

\item[(ii)]  We have
\[
0=g(x)=F_\theta(|x-x_0|) \qquad\iff\qquad |x-x_0|< a
\]
\item[(iii)] Follows from (i) and Proposition \ref{prop0}.
\item[(iv)] We notice that with $S_n\phi(z)=\sum_{k=0}^{n-1}\phi(T^kz)$
\[
\hL^n(f)=\sum_{T^ny=x}\left(e^{S_n\phi(y)}\prod_{k=0}^{n-1}g(T^ky)f(y)\right).
\]
Hence, for all $n\in\N$
\[
\supp\hat\rho=\supp\hL (\hat\rho)=\supp\hL^n(\hat\rho)=\bigcap_{k=0}^{n-1}\supp g\circ T^k.
\]
But
\[
\bigcap_{k=0}^\infty\left(\supp g\circ T^k\right)=\bigcap_{n\in\N}T^{-n}\left([0,1]\setminus (x_0-a,x_0+a)\right).
\]
\item[(v)] This follows from the previous arguments and the fact that the H\"older continuity of $\psi=\log g$ implies that $\supp g=[0,1]$.
\end{itemize}
\end{proof}
Using (iii) of Theorem \ref{main}, we have
\begin{corollary}\label{cor1}
The measure $\hat\alpha$ given by
$$\hat{\alpha}(A):=\frac{1}{\hat{\lambda}}\int_{X}\int_{S}{\bf 1}_A {\bf 1}_{I_{\omega}^c}\hat{\rho}d\theta dm$$
is a conditionally stationary measure which is equivalent to $m$. Its associated escape rate is given by $-\ln\hat\lambda$, where
$$\hat{\lambda}=\int_Sd\theta(\omega)\int_X \hat{\rho}{\bf 1}_{I_{\omega}^c}dm.$$
Moreover, for all $f\in \mathcal H^{\text{loc}}_{\alpha}$ we have
$$\lim_{n\to\infty}||\hat\lambda^{-n}\hL^nf-\hat\rho\int_I fd\hat\nu||_{\infty}=0.$$
\end{corollary}
\begin{proof}
We $\hL \hat\rho=\hat\lambda \hat\rho$, with $\int_I \hat\rho dm=1$. By using the definition of $\hL$, it follows that:
$$
\hat\lambda=\int_{S} d\theta(\omega)\int_I \hat\rho\bold{1}_{I^c_{\omega}}dm.
$$

 By using again the fact that $\hL \hat\rho=\hat\lambda \hat\rho$, with $\int_I \hat\rho dm=1$ and the definition of $\hL$, we see that the above defined measure $\hat\alpha$ satisfies the following:
\begin{equation*}
\hat\alpha(A)=\frac{1}{\hat\lambda}\int_{S} d\theta(\omega)\hat\alpha(T^{-1}A\cap T^{-1}I^c_{\omega});
\end{equation*}
\begin{equation*}
\hat\lambda=\int_{S} d\theta(\omega)\hat\alpha(T^{-1}I^c_{\omega});
\end{equation*}
and
\begin{equation*}
{\hat\lambda}^n\hat\alpha(A)=\int_{\bar S}  d\theta^{\infty}(\bar{\omega})\hat\alpha(T^{-n}A\cap T^{-1}I^c_{\omega_1}\cap T^{-2}I^c_{\omega_2}\cap\cdots\cap T^{-n}I^c_{\omega_n}).
\end{equation*}
The fact that for all $f\in \mathcal H^{\text{loc}}_{\alpha}$ $$\lim_{n\to\infty}||\hat\lambda^{-n}\hL^nf-\hat\rho\int_I fd\hat\nu||_{\infty}=0$$
follows from Theorem \ref{main} and (iv) of Proposition \ref{prop1}.
\end{proof}
\begin{remark}
Since $\hat\alpha$ is equivalent to $m$, by  (i) of Theorem \ref{prop1}, for almost every $\omega$, $m$ almost every $x$ will not survive. This yields to the natural question: what is the `dimension' of the set of points that never escape through the random holes?
\end{remark}

\subsection{Properties of $\hat\mu_{\phi,\theta}$}

The relation between $\hL$ and $\mL_{\phi+\psi}$ yields that the associated measure $\hat\mu_{\phi,\theta}$ is the usual equilibrium state $\mu_{\phi +\psi}$, for the closed deterministic system $T$, with respect to the H\"older continuous potential $\phi +\psi$. It is well known (see for instance \cite{Ba}) that the measure $\mu_{\phi +\psi}$ has the Gibbs property
\[
C^{-1}<\frac{\mu_{\phi +\psi}([x_1\cdots x_n])}{\frac{e^{S_n\phi (x)+S_n\psi (x)}}{\sum_{[y_1\cdots y_n]}e^{S_n\phi (y)+S_n\psi (y)}}}<C
\]
where $C>1$ and $[z_1\cdots z_n]$ denotes the set of all numbers having $z_1\cdots z_n$ as their first dyadic digits. We also call $\hat\mu_{\phi,\theta}$ the measure of conditional surviving probability\footnote{Let us elaborate on calling $\hat\mu_{\phi,\theta}$ the measure of conditional surviving probability. Given a point $y\in [0,1]$ and a finite path of its trajectory $y_1, y_2,\cdots ,y_n$. Then $e^{S_n\phi (y)}$ is up to normalization the probability with respect to the equilibrium state $\mu_\phi$ of choosing this particular path. On the other hand $e^{S_n\psi (y)}$ is up to normalization the probability that $y$ survives along this particular path. Since choosing the path is independent of activating the holes we get
\begin{align*}
\hat\mu_{\phi,\theta}([y_1\cdots y_n])&\sim \frac{\mu_\phi([y_1\cdots y_n])\cdot\P_\theta(y \text{ survives})}{\sum_{[y_1\cdots y_n]}\mu_\phi([y_1\cdots y_n])\cdot\P_\theta(y \text{ survives})}\\&\sim \frac{\mu_\phi([y_1\cdots y_n])\cdot\P_\theta(y \text{ survives})}{\E_{\mu_\phi}(y \text{ survives up till time $n$})}.
\end{align*}
The expectation in the denominator decays exponentially in $n$. Hence the measure $\hat\mu_{\phi,\theta}([y_1\cdots y_n])$ is the conditional probability
of a cylinder to be chosen and surviving rescaled with the average probability of a cylinder of length $n$ to survive under the condition that all trajectories survive until time $n$.}.


\subsection{Asymptotic behaviour and the Hausdoff dimension}
When the set\footnote{Recall that in Remark \ref{non-empty} it is verified that for $T(x)=2x$ the set $\S\ne\emptyset$.} $\S\ne\emptyset$, we consider three families of potentials: $t(\phi +\psi)$, $\phi+t\psi$ and $t\phi +T(t)\psi$ where in the latter case $T(t)$ is defined by $\sup_{\nu -\text{ ergodic}}\{h_\nu+\int_{[0,1]}(t\phi+T(t)\psi)\, d\nu\}=0$. We are interested in the asymptotics of the corresponding equilibrium states as $t\to\infty$.

\begin{theorem}
Any accumulation measure $\mu_\infty^{\S}$ of $\mu_{t(\phi +\psi)}$ as $t\to\infty$ has support on the set of trajectories with the largest possible Birkhoff average of $\phi +\psi$, i.e. for $\mu_\infty^{\S}$-a.e. $x$ the Birkhoff average $\lim_{n\to\infty}\sum_{k=0}^{n-1}\phi(T^k x)+\psi(T^kx)$ is maximal. Moreover it has maximal entropy on this set.
\end{theorem}
\begin{proof}
The statement of the theorem follows from standard results in ergodic optimization \cite{Je}. Any accumulation point maximizes the integral $\int_{[0,1]}(\phi+\psi)\, d\nu$ and is hence supported on the set with maximal possible ergodic averages. Moreover, any such accumulation point has maximal entropy of all invariant measures supported on this set.
\end{proof}

\begin{remark}
The maximal Birkhoff average of $\phi +\psi$ can be interpreted in the following way: the maximal averages are obtained on the set whose generic $\mu_\phi$ trajectory will have the highest joint probability of being chosen and of survival. Then this maximizing measure $\mu_\infty^{\S}$ is concentrated on the set of trajectories where the probability of being chosen and of survivinal is balanced in the highest way. 
\end{remark}

\begin{theorem}
Any accumulation measure $\mu_\infty^{\S}$ of $\mu_{(\phi +t\psi)}$ as $t\to\infty$ has support on $\S$. Moreover,
\[
\sup\{h_\nu+\int_{[0,1]}\phi\, d\nu\, :\, \text{$\nu$- invariant and } \nu(\S)=1\}=h_{\mu_\infty^{\S}}+\int_{[0,1]}\phi\, d\mu_\infty^{\S}.
\]
Therefore it is an equilibrium state with respect to the potential $\phi$ on the surviving set $\S$.
\end{theorem}
\begin{proof}
By the variational principle
\begin{equation*}
\begin{split}
&\sup_{\nu -\text{ invariant}}\left\{h_\nu+\int_{[0,1]}\phi\, d\nu +t\int_{[0,1]}\psi\, d\nu\right\}\\
&\hskip 4cm=h_{\mu_{\phi+t\psi}}+\int_{[0,1]}\phi\, d\mu_{\phi+t\psi} +t\int_{[0,1]}\psi\, d\mu_{\phi+t\psi}.
\end{split}
\end{equation*}
Since $\|\phi\|<\infty$ and $h_\nu\le\log 2$ it follows that $\mu_\infty^{\S}$ maximizes (over all invariant measures) the integral $\int_{[0,1]}\psi\, d\nu$. Let
\[
\tilde{\S}:=\bigcap_{n\in\N}T^{-n}([0,1]\setminus(x_0-b,x_0+b)).
\]
This set is the set of points that never (not only in finite time) enter the critical region and hence have surviving probability 1. We also remark that any invariant measure on $\S$ is indeed concentrated on $\tilde{\S}$.
 Since $\tilde{\S}\ne\emptyset$ and $\tilde{\S}$ is compact and forward invariant there is a maximizing measure $\nu$ supported on $\tilde{\S}$ with $\int_{[0,1]}\psi\, d\nu=0$. Let $B\cap\tilde{\S}=\emptyset$. Then $\mu_\infty^{\S}(B)=0$. Otherwise since $g\vert_B<1$ the integral of $\psi$ would be negative. Hence $\mu_\infty^{\S}(\tilde{\S})=1$ and consequently $\mu_\infty^{\S}(\S)=1$.

Assume that there is an invariant (and hence also an ergodic) measure $\nu$ supported on $\S$ and an $\e >0$ with
\begin{equation}\label{ass}
h_\nu+\int_{[0,1]}\phi\, d\nu > h_{\mu_\infty^{\S}}+\int_{[0,1]}\phi\, d\mu_\infty^{\S}+\e.
\end{equation}
Since $\int_{[0,1]}\psi d\nu=0$, we have
\begin{equation}\label{f1}
h_\nu+\int_{[0,1]}\phi\, d\nu +t\int_{[0,1]}\psi\, d\nu =h_\nu+\int_{[0,1]}\phi\, d\nu.
\end{equation}
Using \eqref{ass} and \eqref{f1}, we get
\begin{equation}\label{f2}
h_\nu+\int_{[0,1]}\phi\, d\nu +t\int_{[0,1]}\psi\, d\nu> h_{\mu_\infty^{\S}}+\int_{[0,1]}\phi\, d\mu_\infty^{\S}+\e.
\end{equation}
Then for sufficiently large $t$, and by using the fact that $t\int_{[0,1]}\psi\, d\mu_{\phi+t\psi}\le 0$, we have
\begin{align*}
h_\nu+\int_{[0,1]}\phi\, d\nu +t\int_{[0,1]}\psi\, d\nu& >h_{\mu_{\phi+t\psi}}+\int_{[0,1]}\phi\, d\mu_{\phi+t\psi}+\frac{\e}{2}\\
&\ge h_{\mu_{\phi+t\psi}}+\int_{[0,1]}\phi\, d\mu_{\phi+t\psi}+t\int_{[0,1]}\psi\, d\mu_{\phi+t\psi}+\frac{\e}{2}.
 \end{align*}
This contradicts the variational principle.
\end{proof}

For the interpretation of the last family we consider the symbolic representation $\Sigma=\{0,1\}^\N$ corresponding to the dyadic expansions of the real numbers in the interval. We introduce a metric on $\Sigma$ by
\[
\diam ([x_1\cdots x_n]):=S_n\tilde\psi (x)
\]
where $\tilde\psi:=\log \frac{g}{\int_{[0,1]}g\, dx}$. Then this space $\Sigma$ becomes a Cantor set. The pointwise dimension $$d_{\phi,\theta}(x)=\frac{\log\mu_\phi([x_1\cdots x_n])}{\diam ([x_1\cdots x_n])}$$ at a point $x$ is a measure of the deviation from the expected dying out probability of a given path.

From standard multifractal analysis \cite{Pe} we know
\begin{theorem}
\[
\dim_H\left\{x:\, \, d_{\phi,\theta}(x)=-T'(t)\right\}=T(t)+T'(t)\cdot t=\dim_H\mu_{t\phi +T(t)\psi}.
\]
\end{theorem}
\section{Examples}
In the previous sections, we studied the dimension of the survival set under the condition that $\theta(\{0\})>0$. To interpret what this means, one can imagine that the random holes are possible `gates' where orbits under the dynamics of the map $T$ escape. The condition $\theta(\{0\})>0$ means that in certain situations \emph{all gates are closed}, and thus no orbit escape. In fact this interpretation explains that the condition $\theta(\{0\})>0$ provides interesting dynamics that cannot happen in a deterministic setting such as in the systems studied in \cite{LM}. Moreover, we believe that the condition $\theta(\{0\})>0$ may provide interesting examples of survival sets that are `fat'. Having said that, it would be also natural to ask what happens when $\theta(\{0\})=0$? We present some examples to shed some light on this direction. Since in the examples below, the measure $\theta$ is discrete, unlike, in Theorem \ref{main}, we first state an analogous result to that of Theorem \ref{main} in a simple discrete case.\\

\noindent $\bullet$ There are $K$ holes, $K\ge 2$, denoted by $\{H_{\omega_i}\}_{i=1}^{K}$ and selected independently according to $\{p_i\}_{i=1}^K$; i.e., $p_i>0$ and $\sum_{i=1}^{K}p_i=1$;\\
$\bullet$ $\{H_{\omega_i}\}_{i=1}^{K-1}$ are elements of the Markov partition of $T$;\\
$\bullet$ $H_K=\emptyset$ (the case of no hole\footnote{This is equivalent to the condition $\theta(\{0\})>0$ in the continuous noise case.});\\
$\bullet$ $\tilde{\mathcal{H}}_{\alpha}^{\text{loc}}:=\{\text{Locally H\"older functions w.r.t. the Markov partition of } T\}$.\\

The following result leeds to an analogous conclusion as that of Theorem \ref{main}.
\begin{proposition}\label{prop2}
Let $\phi\in\tilde{\mathcal{H}}_{\alpha}^{\text{loc}}$. Then for any $f\in\tilde{\mathcal{H}}_{\alpha}^{\text{loc}}$ we have $\hL (f)=\mL_{\phi +\psi} (f)$, with $\psi\in\tilde{\mathcal{H}}_{\alpha}^{\text{loc}}$.
\end{proposition}
\begin{proof}
Let $X_{\omega_i}=[0,1]\setminus H_{\omega_i}$. We have
 $$\hL (f)(x)=\sum_{i=1}^Kp_i\mL_{\phi}(1_{X_{\omega_i}}f)(x)=\mL_{\phi}(\sum_{i=1}^Kp_i1_{X_{\omega_i}}\cdot f)(x)=\mL_{\phi +\psi} (f),$$
 where $\psi(x)=\log\left(\sum_{i=1}^Kp_i1_{X_{\omega_i}}(x)\right)$. Note $\psi\in\tilde{\mathcal{H}}_{\alpha}^{\text{loc}}$ because $\sum_{i=1}^Kp_i1_{X_{\omega_i}}\in\tilde{\mathcal{H}}_{\alpha}^{\text{loc}}$ and $\sum_{i=1}^Kp_i1_{X_{\omega_i}}(x)>0$ since\footnote{In general one can replace the condition $H_K=\emptyset$ by asking $\sum_{i=1}^Kp_i1_{X_{\omega_i}}(x)>0$ and Proposition \ref{prop2} would still hold. We asked explicitly for $H_K=\emptyset$ to be one of the events because this is the most interesting case in our work.} $H_K=\emptyset$.
\end{proof}
\begin{example}\label{ex1}
We consider the following example: $T(x)=2x$ mod $1$. Let $ H=(0,1/2)$ be a {\emph hole}. We now consider {\emph {random holes}} in the following way: at each time $n$, $H$ is either `closed' with probability $p$, or $H$ is `open' with probability $1-p$ . When acting on piecewise constant functions with respect to the partition $[0,1/2),[1/2,1)$, the  transfer operator $\hL$, with $\phi(x)=-\ln|T'x|$, can be represented by the following matrix (acting by multiplication from the left):
\begin{equation}\label{matrix1}
\hat{\mathbb A}_{\phi}:=p\left[\begin{array}{cc}
                1/2 & 1/2 \\
                1/2 & 1/2
              \end{array}\right]+(1-p)\left[\begin{array}{cc}
                0 & 0 \\
                1/2 & 1/2
              \end{array}\right]=\left[\begin{array}{cc}
                p/2 & p/2 \\
                1/2 & 1/2
              \end{array}\right].
\end{equation}
The dominant eigenvalue of the above matrix\footnote{The dominant eigenvalue, and the corresponding eigenfunction, of $\hL$ when acting on H\"older functions, is the same as those of the above matrix reprsentation. The same is true for the dual operator $\hL^*$ using the above matrix with multiplication from the right.}  is given by
$$\hat\lambda=\frac{1+p}{2}.$$ The corresponding left and right eigenvectors are, respectively,
$$\hat\rho=[1\,\,\,1]\hskip 1cm \text{and}\hskip 1cm \hat\nu=[\frac{p}{p+1}\,\,\,\frac{1}{p+1}].$$
Consequently,  for any measurable set $A$, the absolutely continuous conditionally invariant measure $\hat{\alpha}$ is given by:
\begin{equation*}
\begin{split}
\hat{\alpha}(A)&=\frac{1}{\hat{\lambda}}\int_I\int_{\Omega}{\bf 1}_A{\bf 1}_{X_{\omega}}\hat\rho d\theta(\omega) dm\\
&=\frac{2p}{p+1}\left(m(A\cap[0,1/2])\right)+\frac{2}{p+1}\left(m(A\cap[1/2,1])\right).
\end{split}
\end{equation*}
To find the measure $\hat\mu_{\phi,\theta}$, we consider $\hat{ \mathbb A}^*_{\phi}$, the transpose of the matrix defined in \eqref{matrix1}. Let $\mathbb Q$ be the stochastic matrix defined by
$${\mathbb Q}_{ij}=(\hat{\mathbb A}^*_{\phi})_{ij}\frac{\hat\rho_j}{\hat\lambda\rho_i}$$
whose stationary probability vector $q$ is defined by
$$q_{i}=\hat\rho_i\cdot\hat\nu_i;$$
i.e.,
\begin{equation}
\mathbb Q:=\left[\begin{array}{cc}
                p/(p+1) &1/( p+1) \\
                p/(p+1)& 1/(p+1)
              \end{array}\right],
\hskip 1cm q=[p/(p+1)\,\ 1/(p+1)].
\end{equation}
Then $\hat\mu_{\phi,\theta}$, defined on cylinder sets, is given by
\begin{equation*}
\hat\mu_{\phi,\theta}\left(Z(j,a_0,a_1,\dots,a_{n-1}) \right)=q_{a_0}\cdot\mathbb Q_{a_0a_1}\cdots\mathbb Q_{a_{n-2}a_{n-1}}
\end{equation*}
with
$$\hat\mu_{\phi,\theta}\left(Z(j,a_0)\right)=q_{a_0},$$
and each $a_i\in\{1,2\}\equiv\{[0,1/2),[1/2,1)\}$. Notice that
$$(\frac{p}{1+p})^n\le\hat\mu_{\phi,\theta}\left(Z(j,a_0,a_1,\dots,a_{n-1} )\right)\le(\frac{1}{1+p})^n.$$
Thus, we conclude that the support of $\hat\mu_{\phi,\theta}$ is a fat set.
\end{example}
\begin{example}\label{ex2}
We now consider another example to get a sense on the difference between the case `$\theta(\{0\})>0$' (as in Example \ref{ex1}) and the case when `$\theta(\{0\})=0$': $T(x)=3x$ mod $1$. Let $ H_1=(0,1/3)$ and $H_2=(0,2/3)$ be two holes. We now consider random holes in the following way: at each time $n$, we either have a `small' hole $H_1$ with probability $p$, or a bigger hole $H_2$ with probability $1-p$. Again the transfer operator $\hat{\mathcal{L}}_{\phi}$, with $\phi(x)=-\ln|T'x|$, can be represented by the following matrix (acting by multiplication from the left):
\begin{equation}\label{matrix2}
\hat{\mathbb A}_{\phi}:=p\left[\begin{array}{ccc}
                0 & 0 & 0\\
                1/3 & 1/3 & 1/3 \\
                1/3 & 1/3 & 1/3
              \end{array}\right]+(1-p)\left[\begin{array}{ccc}
                0 & 0 &0\\
                0  &0 &0 \\
                1/3&1/3&1/3
              \end{array}\right]=\left[\begin{array}{ccc}
               0 & 0 & 0\\
                p/3 & p/3 & p/3 \\
                1/3 & 1/3 & 1/3
              \end{array}\right].
\end{equation}
The dominant eigenvalue of the above matrix is given by $$\hat\lambda=\frac{1+p}{3}.$$ The corresponding left and right eigenvectors are, respectively,
$$\hat\rho=[1\,\,\,1\,\,\,1]\hskip 1cm \text{and}\hskip 1cm \hat\nu=[0\,\,\,\frac{p}{p+1}\,\,\,\frac{1}{p+1}].$$
Note that for any measurable set $A$, the absolutely continuous conditionally invariant measure $\hat\alpha$ is given by
\begin{equation*}
\begin{split}
\hat\alpha(A)&=\frac{1}{\hat{\lambda}}\int_I\int_{\Omega}{\bf 1}_A{\bf 1}_{X_{\omega}}\hat\rho d\theta(\omega) dm\\
&=\hat{\mu}_{\phi,\theta}(A)=\frac{3p}{p+1}\left(m(A\cap[1/3,2/3])\right)+\frac{3}{p+1}\left(m(A\cap[2/3,1])\right).
\end{split}
\end{equation*}
To find the measure $\hat\mu_{\phi,\theta}$, we consider $\hat{ \mathbb A}^*_{\phi}$, the transpose of the matrix defined in \eqref{matrix2}. Let $\mathbb Q$ be the stochastic matrix defined by
$${\mathbb Q}_{ij}=(\hat{\mathbb A}^*_{\phi})_{ij}\frac{\hat\rho_j}{\hat\lambda\rho_i}$$
whose stationary probability vector $q$ is defined by
$$q_{i}=\hat\rho_i\cdot\hat\nu_i;$$
i.e.,
\begin{equation}
\mathbb Q:=\left[\begin{array}{ccc}
              0&  p/(p+1) &1/( p+1) \\
              0&  p/(p+1) &1/( p+1) \\
               0& p/(p+1)& 1/(p+1)
              \end{array}\right],
\hskip 1cm q=[0\,\ p/(p+1)\,\ 1/(p+1)].
\end{equation}
Then $\hat\mu_{\phi,\theta}$, defined on cylinder sets, is given by
\begin{equation*}
\hat\mu_{\phi,\theta}\left(Z(j,a_0,a_1,\dots,a_{n-1}) \right)=q_{a_0}\cdot\mathbb Q_{a_0a_1}\cdots\mathbb Q_{a_{n-2}a_{n-1}}
\end{equation*}
with
$$\hat\mu_{\phi,\theta}\left(Z(j,a_0)\right)=q_{a_0},$$
and each $a_i\in\{1,2,3\}\equiv\{[0,1/3),[1/3,2/3),[2/3,1)\}$. Notice that
$$0\le\hat\mu_{\phi,\theta}\left(Z(j,a_0,a_1,\dots,a_{n-1}) \right)\le(\frac{1}{1+p})^n.$$
In particular, for any $Z(j,a_0,a_1,\dots,a_{n-1})$ such that $a_i=1$ for some $j\le i\le j+n-1$ we have $\hat\mu_{\phi,\theta}\left(Z(j,a_0,a_1,\dots,a_{n-1})\right)=0$. Thus, we conclude that the support of $\hat\mu_{\phi,\theta}$ is a thin set.
\end{example}
\begin{remark}
The above examples suggest the following interesting problem: In Example \ref{ex1}, where $\theta(\{0\})>0$ the measure $ \hat{\mu}_{\phi,\theta}$ is supported on a fat subset of the interval $[0,1]$. In Example \ref{ex2}, where $\theta(\{0\})=0$, the measure is only supported on a thin subset of $[1/3,1]$. It would be interesting to work out analogous theorems to those in the previous sections in the case when $\theta(0)=0$. Based on the above examples, we believe that for the \emph{same} map $T$, one may get a `fat' survival set when $\theta(\{0\})>0$; however, when $\theta(\{0\})=0$ the survival set would be `thinner'.
\end{remark}

{\bf Acknowledgment.} We would like to thank anonymous referees for their suggestions which improved the presentation of the paper.

\end{document}